\documentclass[10pt, reqno]{amsart}
\usepackage[utf8]{inputenc}
\usepackage[english]{babel}
\usepackage{newunicodechar}
\newunicodechar{́}{\'{}}

\usepackage{amsmath,amssymb,amsfonts}
\usepackage{amsthm}
\usepackage{float}
\usepackage{url}
\usepackage[hidelinks]{hyperref}
\usepackage{MnSymbol} 
\usepackage{booktabs}
    
\theoremstyle{plain}
\newtheorem{theorem}{Theorem}[section]
\newtheorem{prop}[theorem]{Proposition}
\newtheorem{lemma}[theorem]{Lemma}
\newtheorem{corollary}[theorem]{Corollary}
\newtheorem{construction}[theorem]{Construction}

\theoremstyle{definition}
\newtheorem{definition}[theorem]{Definition}

\newtheorem{remark}[theorem]{Remark}

\usepackage{mathtools}
\usepackage{enumitem}
\usepackage{graphicx}
\usepackage{cleveref}
\usepackage[mathscr]{euscript}
\usepackage{tikz}

\title[Multiplicatively idempotent semirings]{Congruence-simple multiplicatively idempotent semirings and multiplicative divisibility}

\author[D. Siejwa]{Damian Siejwa}
\address{Faculty of Mathematics and Information Science\\
Warsaw University of Technology\\
00-661 Warsaw, Poland}

\email{damian.siejwa.dokt@pw.edu.pl}


\begin{document}
\begin{abstract}
    We resolve two conjectures concerning multiplicatively idempotent semirings. First, we prove that every congruence-simple multiplicatively idempotent semiring is finite, and hence belongs to the finite classification of Kepka, Korbelář and Landsmann. Second, we prove that every finitely generated commutative multiplicatively divisible semiring is multiplicatively idempotent. 
\end{abstract}
\maketitle
\section{Introduction}
Semirings provide a natural framework for algebraic structures in which addition need not admit inverses. They arise both in pure algebra and in applications, for instance in automata theory \cite{KuichSalomaa1986, KOSTOLANYI2021101}, graph algorithms \cite{GondranMinoux2008} and dynamic programming \cite{BARIL2026243}. From the structural point of view, one of the fundamental problems is to understand semirings satisfying appropriate simplicity conditions. 

The structure of congruence-simple semirings has been studied in a number of settings. Congruence-simple commutative semirings were classified by El Bashir, Hurt, Jančař\'{ı}k, and Kepka \cite{ELBASHIR2001277}. El Bashir and Kepka later proved structural results for congruence-simple semirings, showing in particular that such semirings fall into three classes: additively idempotent semirings, additively cancellative semirings, additively-nil semirings of index $2$ \cite{ElBashirKepka2007}. In the finite case, Monico proved that congruence-simple semirings belong to one of five types: semirings of order $2$, zero-multiplication rings of prime order, matrix rings over finite fields, additively idempotent semirings, and semirings with trivial addition \cite{MONICO2004846}. The present paper concerns a different restriction on the multiplicative structure, namely multiplicative idempotence. 

Kepka, Korbelář and Landsmann studied congruence-simple multiplicatively idempotent semirings in \cite{KepkaKorbelarLandsmann}. They proved that every such semiring with at least three elements is additively idempotent as well, and hence bi-idempotent. They also classified the finite congruence-simple multiplicatively idempotent semirings, showing that every such semiring is isomorphic to one of the eight two/three-element semirings $\mathbb{S}_1, \ldots, \mathbb{S}_8$ \cite[Theorem 5.6]{KepkaKorbelarLandsmann}. They then asked whether the finite assumption is necessary and conjectured that every congruence-simple multiplicatively idempotent semiring is finite. The reduction to the bi-idempotent case is useful for our approach. By a theorem of Pastijn and Zhao \cite{PastijnZhao}, the multiplicative reduct of every bi-idempotent semiring is a regular band. We use this fact as the starting point for our proof of the conjecture (see Theorem \ref{classification_of_multiplicatively_idempotent_simple_semirings}). 

Kepka, Korbelář and Landsmann also considered multiplicative divisibility. Every multiplicatively idempotent semiring is multiplicatively divisible, and they conjectured that the converse holds for finitely generated commutative semirings. Our second main result proves this conjecture (see Theorem \ref{solution_to_conjecture_2}). The proof first treats the additively cancellative and multiplicatively cancellative cases. For the general case, we pass to a minimal non-idempotent quotient and analyze its multiplicatively cancellative and non-cancellative elements. 

\medskip 
\noindent\textbf{Plan of the paper.} Section~\ref{section_background} contains the necessary preliminaries on semigroups, bands, semirings, semilattices and lattice-ordered groups. In Section~\ref{section_classification_of_simple_multiplicatively_idempotent_semirings} we prove that every congruence-simple multiplicatively idempotent semiring is finite and complete the classification of such semirings. Section~\ref{section_multiplicatively_divisible_commutative_semirings} is devoted to finitely generated commutative multiplicatively divisible semirings. We first consider the additively cancellative and multiplicatively cancellative cases, and then reduce the general case to a minimal non-idempotent quotient. The section concludes with the proof that every finitely generated commutative multiplicatively divisible semiring is multiplicatively idempotent. 
\section{Background}\label{section_background}
In this section we recall basic notions and facts that will be used throughout the paper. Readers familiar with these topics may skip to the next section. 
\subsection{Semigroups and bands} A semigroup $S$ is 
\begin{itemize} 
    \item \emph{cancellative} if both $ax = ay$ and $xa = ya$ imply $x = y$ for all $a, x, y \in S$, 
    \item a \emph{left zero} (resp. \emph{right zero}) \emph{semigroup}  if $xy = x$ (resp. $xy = y$) for all $x, y\in S$, 
    \item a \emph{band} if $x^2 = x$ for every $x\in S$, 
    \item \emph{divisible} if for every $x \in S$ and every natural number $n \geqslant 1$ there exists $y \in S$ such that $y^n = x$. 
\end{itemize}
A band $S$ is a \emph{regular band} if $xyzx = xyxzx$ for all $x, y, z\in S$, and it is a \emph{rectangular band} if $xyx = x$ for all $x, y \in S$. We shall use the following standard characterization of rectangular bands. 
\begin{theorem}\cite[Theorem 1.1.3]{Howie}\label{equivalent_definitions_of_rectangular_band}
    Let $S$ be a semigroup. Then the following conditions are equivalent: 
    \begin{enumerate}[label=(\arabic*)]
        \item $S$ is a rectangular band, 
        \item every element of $S$ is idempotent and $xyz = xz$ for all $x,y,z\in S$, 
        \item there exist a left zero semigroup $L$ and a right zero semigroup $R$ such that $S \cong L \times R$,
        \item $S$ is isomorphic to a semigroup $I \times J$, where $I$ and $J$ are non-empty sets, and where multiplication is given by $(i, j)(k, l) = (i, l)$. 
    \end{enumerate}
\end{theorem}
\subsection{Semirings} A \emph{semiring} is a non-empty set $R$ with two binary operations $+: R \times R \to R$ and $\cdot: R \times R \to R$ (called \emph{addition} and \emph{multiplication}) such that $(R, +)$ is a commutative semigroup, $(R, \cdot)$ is a semigroup and $$r \cdot (s + t) = r \cdot s + r \cdot t, \quad (r + s)\cdot t = r \cdot t + s \cdot t$$ for all $r, s, t \in R$. For notational convenience, we write $rs$ instead of $r \cdot s$ for $r, s \in R$. A~\emph{semiring congruence} on $R$ is an equivalence relation $\sim\: \subseteq R \times R$ such that $$r \sim s \implies r + t \sim s + t, \; tr \sim ts, \text{ and } rt \sim st$$ for all $r, s, t \in R.$ We denote by $\mathrm{Con}(R)$ the set of all congruences on $R$, ordered by inclusion. For a subset $A\subseteq R \times R$ we write $\langle A \rangle$ for the least congruence on $R$ containing $A$. If $\theta\in\mathrm{Con}(R)$, then $R/\theta$ denotes the corresponding quotient semiring. A semiring $R$ is 
\begin{itemize}
    \item \emph{trivial} if it has exactly one element, 
    \item \emph{commutative} if its multiplication is commutative, 
    \item \emph{additively} (resp. \emph{multiplicatively}) \emph{cancellative} if its additive (resp. multiplicative) reduct is cancellative, 
    \item \emph{congruence-simple} if it has exactly two congruences, namely the \emph{equality congruence} $\mathrm{id}_R$ and the \emph{universal congruence} $R \times R$, 
    \item \emph{additively} (resp. \emph{multiplicatively}) \emph{idempotent} if $r + r = r$ (resp. $r^2 = r$) for every $r \in R$, 
    \item \emph{bi-idempotent} if it is both additively and multiplicatively idempotent, 
    \item a \emph{parasemifield} if its multiplicative reduct is a group, 
    \item \emph{multiplicatively divisible} if its multiplicative reduct is a divisible semigroup.
\end{itemize}
An element $a \in R$ is \emph{left} (resp. \emph{right}) \emph{absorbing} if $ar = a$ (resp. $ra = a$) for every $r \in R$. We shall also use the following result of Pastijn and Zhao. 
\begin{theorem}\cite[Theorem 2.3]{PastijnZhao}\label{multiplicative_reduct_of_bi_idempotent_semiring_is_regular_band}
    Let $(R, +, \cdot)$ be a bi-idempotent semiring. Then the multiplicative reduct $(R, \cdot)$ is a regular band. 
\end{theorem}
\subsection{Semilattices and lattice-ordered groups}
A \emph{semilattice} is a partially ordered set $L$ in which every two elements $x,y\in L$ have a supremum, denoted by $x \lor y$. A \emph{lattice} is a partially ordered set in which every two elements have both a supremum $x\lor y$ and an infimum $x\land y$. 

If $R$ is an additively idempotent semiring, then addition induces a partial order $\leqslant\: \subseteq R \times R$ by $r \leqslant s \iff r + s = s$. With respect to this order, $r \lor s = r + s$, so the additive reduct of $R$ may be regarded as a semilattice. 

An \emph{$\ell$-group} is a group $G$ equipped with a lattice order $\leqslant\subseteq G \times G$ such that $$x \leqslant y \implies axb \leqslant ayb$$ for all $x, y, a, b \in G$. 
\begin{lemma}\label{divisible_abelian_finitely_generated_l_group_is_trivial}
    Let $G$ be a divisible abelian $\ell$-group. If $G$ is finitely generated as an $\ell$-group, then $G$ is trivial.
\end{lemma}
\begin{proof}
    By \cite[Corollary 1]{GlassMarra}, the underlying abelian group of $G$ is free. Suppose that $G$ is non-trivial. Let $B$ be a basis of $G$. Fix $b \in B$. Since $G$ is divisible, there exists $x\in G$ such that $2x = b$. We may write $x = n_1b_1 + \ldots + n_kb_k$, where $n_1, \ldots, n_k \in\mathbb{Z}$ and $b_1, \ldots, b_k\in B$ are pairwise distinct. Then $b = 2n_1b_1 + \ldots + 2n_kb_k$. By uniqueness of the representation of an element of a free abelian group with respect to the basis $B$, the coefficient of $b$ on the right-hand side must be $1$. However, every coefficient on the right-hand side is even, a contradiction. Therefore, $G$ is trivial. 
\end{proof}
\section{Classification of congruence-simple multiplicatively idempotent semirings}\label{section_classification_of_simple_multiplicatively_idempotent_semirings}
In this section we complete the classification of congruence-simple multiplicatively idempotent semirings. In \cite{KepkaKorbelarLandsmann} Kepka, Korbelář and Landsmann listed eight semirings $\mathbb{S}_1, \ldots, \mathbb{S}_8$, whose operation tables are as follows. 
$$
\begin{array}{c@{\qquad\qquad}c}
\mathbb{S}_1:
\quad
\begin{array}{c|cc}
+ & w & a \\ \hline
w & w & a \\
a & a & w
\end{array}
\quad
\begin{array}{c|cc}
\cdot & w & a \\ \hline
w & w & w \\
a & w & a
\end{array}
&
\mathbb{S}_2:
\quad
\begin{array}{c|cc}
+ & w & a \\ \hline
w & w & a \\
a & a & a
\end{array}
\quad
\begin{array}{c|cc}
\cdot & w & a \\ \hline
w & w & w \\
a & w & a
\end{array}
\end{array}
$$

$$
\begin{array}{c@{\qquad\qquad}c}
\mathbb{S}_3:
\quad
\begin{array}{c|cc}
+ & w & a \\ \hline
w & w & w \\
a & w & a
\end{array}
\quad
\begin{array}{c|cc}
\cdot & w & a \\ \hline
w & w & w \\
a & w & a
\end{array}
&
\mathbb{S}_4:
\quad
\begin{array}{c|cc}
+ & w & a \\ \hline
w & w & w \\
a & w & w
\end{array}
\quad
\begin{array}{c|cc}
\cdot & w & a \\ \hline
w & w & w \\
a & w & a
\end{array}
\end{array}
$$

$$
\begin{array}{c@{\qquad\qquad}c}
\mathbb{S}_5:
\quad
\begin{array}{c|cc}
+ & a & w \\ \hline
a & a & w \\
w & w & w
\end{array}
\quad
\begin{array}{c|cc}
\cdot & a & w \\ \hline
a & a & a \\
w & w & w
\end{array}
&
\mathbb{S}_6:
\quad
\begin{array}{c|cc}
+ & a & w \\ \hline
a & a & w \\
w & w & w
\end{array}
\quad
\begin{array}{c|cc}
\cdot & a & w \\ \hline
a & a & w \\
w & a & w
\end{array}
\end{array}
$$

$$
\mathbb{S}_7:
\quad
\begin{array}{c|ccc}
+ & a & b & w \\ \hline
a & a & b & w \\
b & b & b & w \\
w & w & w & w
\end{array}
\quad
\begin{array}{c|ccc}
\cdot & a & b & w \\ \hline
a & a & a & a \\
b & a & b & w \\
w & w & w & w
\end{array}
$$

$$
\mathbb{S}_8:
\quad
\begin{array}{c|ccc}
+ & a & b & w \\ \hline
a & a & b & w \\
b & b & b & w \\
w & w & w & w
\end{array}
\quad
\begin{array}{c|ccc}
\cdot & a & b & w \\ \hline
a & a & a & w \\
b & a & b & w \\
w & a & w & w
\end{array}
$$
Then they proved the following theorem. 
\begin{theorem}\cite[Theorem 5.6]{KepkaKorbelarLandsmann}\label{classification_of_finite_multiplicatively_idempotent_simple_semirings} 
    Let $R$ be a finite multiplicatively idempotent congruence-simple semiring. Then $R$ is isomorphic to one of the eight two/three-element semirings $\mathbb{S}_1, \ldots, \mathbb{S}_8$. 
\end{theorem} 
Kepka, Korbelář and Landsmann further asked whether the finiteness assumption is necessary. We prove that it is not: every congruence-simple multiplicatively idempotent semiring is finite. We begin by recalling one result, which reduces the problem to the bi-idempotent case. 
\begin{prop}\cite[Proposition 4.1]{KepkaKorbelarLandsmann}\label{multiplicatively_idempotent_simple_semiring_is_additively_idempotent}
    Let $R$ be a multiplicatively idempotent congruence-simple semiring containing at least three elements. Then $R$ is bi-idempotent. 
\end{prop} 
\begin{lemma}\label{two_trivial_congruences}
    Let $R$ be a bi-idempotent congruence-simple semiring containing at least three elements. Define equivalence relations $$\lambda = \big\{(x, y)\in R \times R \mid rx = ry \text{ for every } r \in R\big\}$$ and $$\rho = \big\{(x, y)\in R \times R \mid xr = yr \text{ for every } r \in R\big\}.$$ Then $\lambda = \rho = \mathrm{id}_R$. 
\end{lemma}
\begin{proof}
    It can be easily seen that $\lambda$ is a semiring congruence. Indeed, if $(x, y) \in \lambda$, then for all $r, t\in R$, 
    \begin{equation*}\begin{aligned}
        r(x + t) &= rx + rt = ry + rt = r(y +t), \\ 
        r(xt) &= (rx)t = (ry)t = r(yt), \\ 
        r(tx) &= (rt)x = (rt)y = r(ty). 
    \end{aligned}\end{equation*}
    Since $R$ is congruence-simple, $\lambda = \mathrm{id}_R$ or $\lambda = R \times R$. Suppose, for the sake of contradiction, that $\lambda = R \times R$. Then for all $x, y\in R$, we have $xy = xx = x$ and hence every element is left absorbing. Since the additive reduct $(R, +)$ is a semilattice and $|R| \geqslant 3$, we can choose an element $a \in R$ which is neither least nor greatest with respect to the order induced by addition. Now define an equivalence relation $\theta \subseteq R \times R$ by $$(x, y)\in\theta\iff x + a = y + a.$$ We prove that it is a semiring congruence. Indeed, compatibility with addition is trivial. Moreover, if $(x, y)~\in~\theta$, then for every $r \in R$, $xr = x$, $yr = y$, $rx = r$, and $ry = r$, so $(xr, yr), (rx, ry) \in\theta.$ Since the element $a$ is not least, there exists $s \in R$ such that $s + a \neq s$. As $(s + a, s) \in \theta$, we get $\theta \neq \mathrm{id}_R$. Similarly, since $a$ is not greatest, there exists $t \in R$ such that $t + a \neq a$ and in consequence $(t + a, a) \not\in\theta$. Taken together, $\theta$~is non-trivial proper semiring congruence, contradicting simplicity of $R$. Hence $\lambda = \mathrm{id}_R$ and analogously we show that also $\rho = \mathrm{id}_R$.
\end{proof}
\begin{lemma}\label{multiplication_by_the_same_element_from_both_sides_gives_endomorphism}
    Let $R$ be a semiring whose multiplicative reduct is a regular band. Then for every $r \in R$, the map $$h_r : R \to R, \quad h_r(x) = rxr$$ is a semiring endomorphism. 
\end{lemma}
\begin{proof}
    Let $r, x, y \in R$. Using distributivity, idempotence, and the regular band identity $xyzx = xyxzx$, we obtain 
    \begin{align}
        h_r(x + y) &= r(x + y)r = rxr + ryr = h_r(x) + h_r(y), \notag\\ 
        h_r(xy) &= rxyr = rxryr = rxrryr = h_r(x)h_r(y). \tag*{\qedhere}
    \end{align}
\end{proof}
\begin{prop}\label{every_element_is_multiplicative_identity_or_satisfies_certain_equality}
    Let $R$ be a congruence-simple semiring whose multiplicative reduct is a regular band. If $|R| \geqslant 2$, then for every element $r \in R$ exactly one of the following alternatives holds: 
    \begin{enumerate}[label=(\arabic*)]
        \item $r$ is a multiplicative identity of $R$, 
        \item $rxr = r$ for every $x \in R$.
    \end{enumerate}
\end{prop}
\begin{proof}
    Let $r \in R$. It can be easily seen that, if $|R| \geqslant 2$, then conditions $(1)$ and $(2)$ are mutually exclusive. By Lemma \ref{multiplication_by_the_same_element_from_both_sides_gives_endomorphism}, $$h_r : R \to R, \quad h_r(x) = rxr$$ is an endomorphism of $R$. Hence its kernel is a semiring congruence. Since $R$ is congruence-simple, $\ker(h_r) = \mathrm{id}_R$ or $\ker(h_r) = R \times R$. 

    First suppose that $\ker(h_r) = \mathrm{id}_R$. Then $h_r$ is injective. Since $$h_r\big(h_r(x)\big) = h_r(rxr) = rrxrr = rxr = h_r(x),$$ we have $h_r(x) = x$, so $h_r$ is the identity map. Thus, $rxr = x$ for every $x \in R$. Applying $h_r$ to $rx$ gives $$rx = h_r(rx) = rrxr = rxr = x.$$ Similarly, we apply $h_r$ to $xr$ and obtain $$xr = h_r(xr) = rxrr = rxr = x.$$ Therefore, $r$ is a multiplicative identity of $R$. 

    Now suppose that $\ker(h_r) = R \times R$. Then $h_r$ is constant and in consequence, $$rxr = h_r(x) = h_r(r) = rrr = r$$ for every $x \in R$. 
\end{proof}
\begin{prop}\label{R_has_a_multiplicative_identity}
    Let $R$ be a bi-idempotent congruence-simple semiring containing at least three elements. Then $R$ has a multiplicative identity. 
\end{prop}
\begin{proof}
    By Theorem \ref{multiplicative_reduct_of_bi_idempotent_semiring_is_regular_band}, the multiplicative reduct of $R$ is a regular band. Suppose, for the sake of contradiction, that $R$ has no multiplicative identity. Then by Proposition \ref{every_element_is_multiplicative_identity_or_satisfies_certain_equality}, the equality $rxr = r$ holds for all $r, x \in R$. Hence the multiplicative reduct of $R$ is a rectangular band. Let $x, y \in R$. By Theorem \ref{equivalent_definitions_of_rectangular_band}, for every $r \in R$, $$r(xy) = rxy = ry,$$ so $(xy, y)\in\lambda.$ By Lemma \ref{two_trivial_congruences}, $\lambda = \mathrm{id}_R$ and in consequence $xy = y$. Similarly, for every $r \in R$, $$(xy)r = xyr = xr,$$ so $(xy, x)\in\rho$. Again by Lemma \ref{two_trivial_congruences}, $\rho = \mathrm{id}_R$ and hence $xy = x$. Thus, $x = y$ for all $x, y \in R$, which leads to a contradiction with $|R| \geqslant 3$. Therefore, $R$ has a multiplicative identity.  
\end{proof}
We now complete the classification. The idea is as follows: once a multiplicative identity exists, the set of non-identity elements forms a rectangular band. It turns out that simplicity of $R$ forces this rectangular band to be either a left zero band or a right zero band. Using distributivity we then show that it has at most two elements.
\begin{theorem}\cite[Conjecture 1]{KepkaKorbelarLandsmann}\label{classification_of_multiplicatively_idempotent_simple_semirings}
    Let $R$ be a multiplicatively idempotent congruence-simple semiring. Then $R$ is finite (and isomorphic to one of the semirings $\mathbb{S}_1, \ldots, \mathbb{S}_8$). 
\end{theorem}
\begin{proof}
    If $|R| \leqslant 2$, then $R$ is finite. Assume therefore that $|R| \geqslant 3$. By Proposition \ref{multiplicatively_idempotent_simple_semiring_is_additively_idempotent}, the semiring $R$ is bi-idempotent. By Proposition \ref{R_has_a_multiplicative_identity}, the semiring $R$ has a multiplicative identity $1$. Put $S := R \setminus\{1\}$. We claim that $S$ is a rectangular band. If $x, y \in S$ and $xy = 1$, then $$x = x1 = xxy = xy = 1,$$ which leads to a contradiction. Hence $S$ is closed under multiplication. By Proposition \ref{every_element_is_multiplicative_identity_or_satisfies_certain_equality}, the equality $rxr = r$ holds for all $r,x \in S$, so $S$ is a rectangular band. 

    By Theorem \ref{equivalent_definitions_of_rectangular_band}, $S$ is isomorphic to a semigroup $I \times J$, where $I$ and $J$ are non-empty sets and multiplication is given by $(i,j)(k,l) = (i,l)$. We show that one of the sets $I, J$ is a singleton. Assume towards a contradiction that $|I|, |J| \geqslant 2$. Define an equivalence relation $\eta\subseteq R \times R$ by putting 
    \begin{equation*}\begin{aligned}
        &[1] = \{1\}, \\
        &\big((i,j), (k,l)\big)\in\eta \iff i = k  
    \end{aligned}\end{equation*}
    for all $(i, j),(k,l)\in S$, where $[1]$ denotes the equivalence class of the multiplicative identity $1$. We prove that it is a semiring congruence. Indeed, compatibility with multiplication is trivial. Let $x = (i, j)$, $y = (i, j')$ and let $z\in R$. For every $r \in S$, distributivity gives $(x + z)r = xr + zr$ and $(y + z)r = yr + zr.$ Since $xr = yr$, we have $$(x + z)r = (y + z)r$$ for every $r \in S$. Thus, $(x + z, y + z)\in\eta$, as required. Therefore, $\eta\subseteq R \times R$ is a non-trivial proper semiring congruence, which leads to a contradiction with simplicity of $R$. Hence one of the sets $I, J$ is a singleton.

    We conclude that $S$ is either a left zero band or a right zero band. Without loss of generality, assume that $S$ is a left zero band, so $xy = x$ for all $x, y\in S$. Now fix elements $x, y \in S$. Since $R = S \cup\{1\}$, either $1 + x = 1$ or $1 + x\in S$. If $1 + x = 1$, then multiplying both sides by $y$ on the right gives $y + x = y$. Thus, $x \leqslant y$ and $x \leqslant 1$. Hence $x$ is the least element of $R$. In the second case, assume $1 + x \in S$. Then $$1 + x = (1 + x)x = x + xx = x,$$ so $x \geqslant 1$. Moreover, 
    $$x = 1 + x = (1 + x)y = y + xy = y + x$$ and thus, $x \geqslant y$. Hence $x$ is the greatest element of $R$.
    
    Therefore, every element of $S$ is either the least element of $R$ or the greatest element of $R$. Since a poset has at most one least element and at most one greatest element, it follows that $|S| \leqslant 2$ and in consequence $|R| \leqslant 3$. By Theorem \ref{classification_of_finite_multiplicatively_idempotent_simple_semirings}, $R$ is isomorphic to one of the semirings $\mathbb{S}_1, \ldots, \mathbb{S}_8$. This completes the proof. 
\end{proof}

\section{Multiplicatively divisible commutative semirings}\label{section_multiplicatively_divisible_commutative_semirings}
\subsection{The additively cancellative case}
\begin{definition}\cite[Section 1]{ORZECH197081}
    A ring $R$ is called \emph{residually finite} if for every nonzero element $r \in R$ there exists an ideal $I$ of $R$ such that $r\not\in I$ and $R/I$ is finite. 
\end{definition}
\begin{theorem}\cite[Theorem 1]{ORZECH197081}\label{finitely_generated_commutative_ring_is_residually_finite}
    Let $R$ be a finitely generated commutative ring. Then $R$ is residually finite. 
\end{theorem}
\begin{remark}\cite[Theorem 2.5]{ClarkHollandSzekely}\label{finite_divisible_semigroup_is_a_finite_band}
    Let $S$ be a divisible semigroup. Then $S$ is a finite band in each of the following two cases: 
    \begin{enumerate}[label=(\arabic*)]
        \item $S$ is commutative and finitely generated, 
        \item $S$ is finite.
    \end{enumerate}
\end{remark}
\begin{construction}\label{construction_of_completion_group}
    Let $(S, +)$ be a cancellative commutative semigroup. Define an equivalence relation $\tau\subseteq(S\times S) \times (S\times S)$ by $$\big((x, y), (z, w)\big)\in\tau \iff x + w = y + z.$$ Let $G(S) := (S \times S)/\tau.$ Denote by $[x, y]$ the equivalence class of $(x, y)$ in $G(S)$. Now define addition and inverse by $$[x, y] + [z, w] := [x + z, y + w], \quad -[x, y] := [y,x].$$ This makes $G(S)$ an abelian group. Its identity element is given by $0 := [x, x]$ for any $x \in S$.

    Fix $e \in S$ and define the map $$\iota: S \to G(S), \quad \iota(s) = [s + e, e].$$ Then $\iota$ is an injective semigroup homomorphism. Indeed, if $$[s + e, e] = [t + e, e],$$ then $$(s + e) + e = (t + e) + e,$$ so cancellation gives $s = t$. Moreover, $$\iota(s) + \iota(t) = [s + e, e] + [t + e, e] = [s + t + 2e, 2e] = [s + t + e, e] = \iota(s + t).$$ Finally, every element of $G(S)$ is a difference of two elements from $\iota(S)$, because $$[s, t] = \iota(s) - \iota(t).$$
\end{construction}
\begin{prop}\label{finitely_generated_commutative_multiplicatively_divisible_additively_cancellative_semiring_is_multiplicatively_idempotent}
    Let $R$ be a finitely generated commutative semiring. If $R$ is multiplicatively divisible and additively cancellative, then $R$ is multiplicatively idempotent. 
\end{prop}
\begin{proof}
    Since the reduct $(R, +)$ is a cancellative commutative semigroup, we apply Construction \ref{construction_of_completion_group} to obtain the abelian group $(G(R), +)$ and the injective semigroup homomorphism $\iota: R \to G(R)$. Now define multiplication in $G(R)$ by $$[x, y] \cdot [z, w] := [xz + yw, xw + yz].$$ This operation is well defined, since multiplication in $R$ is distributive. Thus, $G(R)$ becomes a commutative ring. The map $\iota: R \to G(R)$ preserves multiplication as well, because $$\iota(r)\cdot\iota(s) = [r + e, e] \cdot [s + e, e] = [rs + re + se + 2e^2, re + se + 2e^2] = [rs + e, e] = \iota(rs).$$ Thus, the map $\iota : R \to G(R)$ is an embedding and $R$ may be viewed as a subsemiring of $G(R)$.

    Recall that each element of $G(R)$ is of the form $[x, y] = \iota(x) - \iota(y)$. Therefore, if $x_1, \ldots, x_k$ generate $R$ as a semiring, then $G(R)$ is generated as a ring by $\iota(x_1), \ldots, \iota(x_k)$. Now suppose that $R$ is not multiplicatively idempotent. Hence there exists $r \in R$ such that $r^2 \neq r$. Since $\iota$ is injective, we have $\iota(r)^2 - \iota(r) \neq 0$ in $G(R)$. By Theorem \ref{finitely_generated_commutative_ring_is_residually_finite}, the ring $G(R)$ is residually finite, so there exists an ideal $I$ of $G(R)$ such that $\iota(r)^2 - \iota(r) \not\in I$ and $G(R)/I$ is finite. Let $\pi: G(R) \to G(R)/I$ be the natural quotient map. Since $\iota(r)^2 - \iota(r) \not\in I$, we have $\pi(\iota(r))^2 \neq \pi(\iota(r))$. Now consider the subset $$T := \pi\big(\iota(R)\big) \subseteq G(R)/I.$$ This is a finite commutative multiplicative semigroup, because $G(R)/I$ is finite. Moreover, $T$ is divisible. Indeed, if $t\in T$, then $t = \pi(\iota(s))$ for some $s \in R$. Since $R$ is multiplicatively divisible, for every $n \geqslant 1$ there exists $q \in R$ such that $s = q^n$. Therefore, $$t = \pi\big(\iota(s)\big) = \pi\big(\iota(q^n)\big) = \pi\big(\iota(q)\big)^n.$$ By Remark \ref{finite_divisible_semigroup_is_a_finite_band}, $T$ is a band. Therefore, every element of $T$ is idempotent. In particular, $\pi(\iota(r))^2 = \pi(\iota(r))$, which leads to a contradiction. Thus, $R$ is multiplicatively idempotent. 
\end{proof}
\subsection{The multiplicatively cancellative case} 
\begin{prop}\label{finitely_generated_commutative_multiplicatively_divisible_multiplicatively_cancellative_semiring_is_trivial}
    Let $R$ be a finitely generated commutative semiring. If $R$ is multiplicatively divisible and multiplicatively cancellative, then $R$ is trivial.  
\end{prop}
\begin{proof}
    Since the reduct $(R, \cdot)$ is a cancellative commutative semigroup, we apply Construction \ref{construction_of_completion_group} to obtain the abelian group $(G(R), \cdot)$ and the injective semigroup homomorphism $\iota: R \to G(R)$. Note that now the group operation of $G(R)$ is written multiplicatively. From now on, we identify $R$ with $\iota(R)\subseteq G(R)$. Thus, multiplication in $R$ is the restriction of the group operation of $G(R)$. Now we extend the semiring addition $+$ from $R$ to $G(R)$. Let $x, y \in G(R)$. Then $x = rs^{-1}$ and $y = tu^{-1}$ for some $r, s, t, u \in R$. We have $xsu = ru \in R$ and $ysu = ts\in R$. Define $$x + y := (xsu + ysu)(su)^{-1}.$$ More generally, if $w \in R$ satisfies $xw, yw \in R$, then $x + y = (xw + yw)w^{-1}$. The value of $x + y$ is independent of the chosen $w$. Indeed, if $w, v \in R$ satisfy $xw, yw, xv, yv \in R$, then after multiplying $(xw + yw)w^{-1}$ and $(xv + yv)v^{-1}$ by $wv$ we get $$(xw + yw)v = xwv + ywv$$ and $$(xv + yv)w = xvw + yvw,$$ which are equal by commutativity of multiplication in $G(R)$. Since $G(R)$ is a group, cancellation gives $$(xw + yw)w^{-1} = (xv + yv)v^{-1}.$$ Thus, the extended operation is well defined. Moreover, it can be easily verified that the extended operation is associative and commutative, coincides with the original addition on $R$, and satisfies 
    \begin{equation}\label{distributivity_property}
        (x + y)z = (xz) + (yz)  
    \end{equation}
    for all $x, y, z \in G(R)$.
    We check only distributivity. Write $x = rs^{-1}$, $y = tu^{-1}$, $z = wv^{-1}$ for some $r, s, t, u, w, v \in R$ and put $h := suv$. Then $xh, yh, zh, xzh, yzh \in R$. Multiplying both sides of \eqref{distributivity_property} by $h^2$, the left-hand side becomes $$\big((x + y)z\big)h^2 = \big((x + y)h\big)(zh) = (xh + yh)(zh) = xzh^2 + yzh^2,$$ and the right-hand side becomes $$\big((xz) + (yz)\big)h^2 = (xzh + yzh)h = xzh^2 + yzh^2.$$ Using cancellation, we get \eqref{distributivity_property}, as required. Hence $G(R)$ is a commutative parasemifield. 

    Now define an equivalence relation $\alpha\subseteq G(R) \times G(R)$ by $$(x, y)\in\alpha \iff x + c = y + c$$ for some $c\in G(R)$. It is compatible with both semiring operations of $G(R)$ and in consequence $G(R)/\alpha$ is a commutative semiring. We claim that the addition in $G(R)/\alpha$ is cancellative. Indeed, if $[x] + [z] = [y] + [z]$ in $G(R)/\alpha$, then $(x + z, y + z)\in\alpha$, so $(x + z) + c = (y + z) + c$ for some $c \in G(R)$. Hence $x + (z + c) = y + (z + c)$, and therefore $(x, y)\in\alpha$. 
    
    Let $\overline{R}$ be the image of $R$ in $G(R)/\alpha$. Then $\overline{R}$ is finitely generated, multiplicatively divisible and additively cancellative. By Proposition \ref{finitely_generated_commutative_multiplicatively_divisible_additively_cancellative_semiring_is_multiplicatively_idempotent}, $\overline{R}$ is multiplicatively idempotent. But multiplication in $G(R)/\alpha$ is a group operation, so its only idempotent element is the identity. Hence every element of $\overline{R}$ is the identity of $G(R)/\alpha$. Since every element of $G(R)$ is of the form $rs^{-1}$, where $r, s\in R$, it follows that $G(R)/\alpha$ is trivial. Thus, $\alpha = G(R) \times G(R)$. 
    
    We claim that $G(R)$ is trivial. Define the map $$f: G(R) \to G(R), \quad f(x) := 1 + x.$$ Since $\alpha = G(R)\times G(R)$, there exists $c \in G(R)$ such that $1 + c = (1 + 1) + c$. Thus, we have $f(f(c)) = f(c)$, so the set $$P := \{x \in G(R) \mid f(x) = x\}$$ is non-empty. Define $$J := \{x \in G(R) \mid Px \subseteq P\}, \quad I := J \cap J^{-1} = \{x \in G(R) \mid Px = P\},$$ where $J^{-1} := \{x^{-1} \mid x \in J\}$. Then $J$ is a submonoid of $G(R)$ and $I$ is a subgroup of $G(R)$. Note that $f(G(R))\subseteq J$. Indeed, if $p\in P$ and $x\in G(R)$, then $$pf(x) = p + px = (1 + p) + px = 1 + (p + px) = 1 + pf(x).$$ Thus, $pf(x)\in P$, so $f(x) \in J$. 

    Now define a partial order on $G(R)/I$ by $$xI \leqslant yI \iff yx^{-1}\in J.$$ This is well defined. Indeed, since $I^{-1} = I \subseteq J$, for all $i, k\in I$ we have $$xI \leqslant yI \iff yx^{-1}\in J \iff yx^{-1}ki^{-1} \in J\iff (xi)I \leqslant (yk)I.$$ Moreover, this is translation-invariant: if $xI \leqslant yI$, then for any $z \in G(R)$, $$(yz)(xz)^{-1} = yzz^{-1}x^{-1} = yx^{-1} \in J$$ and in consequence $(xz)I \leqslant (yz)I$. Now we prove that $G(R)/I$ is an $\ell$-group. For $x\in G(R)$, the join of $I$ and $xI$ is $f(x)I$. Indeed, we have $f(x)\in J$, so $I \leqslant f(x)I$. Moreover, $f(x)x^{-1} = f(x^{-1})\in J$ and thus, $xI \leqslant f(x)I$. Now let $uI$ be any upper bound of $I$ and $xI$. Then we have $u, ux^{-1}\in J$. Let $p\in P$ and put $q := pux^{-1}$. It is easy to see that $q\in P$ and $qx = pu \in P$. Since $q \in P$, we get  $1 + q = q$ and multiplying both sides by $x$ gives $x + qx = qx$. Since $qx\in P$, we also have $1 + qx = qx$. Therefore, $$f(x)f\big(f(x)^{-1}qx\big) = f(x) + qx = (1 + x) + qx = 1 + (x + qx) = 1 + qx = qx$$ and multiplying both sides by $f(x)^{-1}$ on the left gives $$f\big(f(x)^{-1}qx\big) = f(x)^{-1}qx.$$ Then $f(x)^{-1}qx \in P$. Since $G(R)$ is abelian, $$puf(x)^{-1} = f(x)^{-1}qx \in P.$$ But $p\in P$ was chosen arbitrarily, so $uf(x)^{-1}\in J$. Therefore, $f(x)I \leqslant uI$ and in consequence $I \lor xI = f(x)I$, as required. By translation invariance of the order on $G(R)/I$, we obtain $$xI \lor yI = xI \lor xx^{-1}yI = x(I \lor x^{-1}yI) = xf(x^{-1}y)I = (x + y)I$$ and meets are given by $$xI \land yI = (x^{-1}I\lor y^{-1}I)^{-1}.$$ Thus, $G(R)/I$ is an abelian $\ell$-group. 

    Let $x_1, \ldots, x_k$ be the generators of semiring $R$. Their images generate $G(R)/I$ as an $\ell$-group. Moreover, $G(R)/I$ is divisible. Indeed, let $xI \in G(R)/I$ and let $n \geqslant 1$. Since each element of $G(R)$ can be written in the form $rs^{-1}$, where $r, s\in R$, we have $x = tu^{-1}$ for some $t, u\in R$. By the multiplicative divisibility of $R$, there exist $w, v\in R$ such that $t = w^n$ and $u = v^n$. Thus, $$(wv^{-1}I)^n = (wv^{-1})^nI = w^n(v^n)^{-1}I = tu^{-1}I = xI,$$ which proves divisibility of $G(R)/I$. By Lemma \ref{divisible_abelian_finitely_generated_l_group_is_trivial}, the group $G(R)/I$ is trivial. Thus, $I = G(R)$. Since $P \neq \varnothing$ and $I = \{x \in G(R) \mid Px = P\}$, we get $PG(R) = PI = P$ and $PG(R) = G(R)$. Therefore, $P = G(R)$ and in consequence $f(x) = x$ for every $x \in G(R)$. Then $$x = f(x) = x\big(f(x^{-1})\big) = xx^{-1} = 1,$$ so $G(R)$ is trivial. Since the map $\iota: R \to G(R)$ is injective, the semiring $R$ must also be trivial. 
\end{proof}
\subsection{Minimal non-idempotent quotients}
\begin{lemma}\label{finitely_generated_commutative_multiplicatively_idempotent_semiring_is_finite}
    Let $R$ be a finitely generated commutative semiring. If $R$ is multiplicatively idempotent, then $R$ is finite.
\end{lemma}
\begin{proof}
    Let $X$ be a finite generating set of $R$. Since multiplication is commutative and idempotent, there are only finitely many products of elements in $X$. For such a product $p$, we have $$2p = p + p = (p + p)^2 = p^2 + p^2 + p^2 + p^2 = 4p,$$ so the additive cyclic semigroup generated by $p$ is finite. Every element of $R$ is a finite sum of finitely many products of elements in $X$ and only finitely many such sums occur. 
\end{proof}
\begin{definition}
    A semiring $R$ is called \emph{minimal non-idempotent} if $R$ is not multiplicatively idempotent, whereas every proper quotient of $R$ is. 
\end{definition}
\begin{lemma}\label{existence_of_minimal_non_idempotent_quotient}
    Let $R$ be a finitely generated commutative semiring which is not multiplicatively idempotent. Then $R$ has a quotient $T$ which is finitely generated, commutative and minimal non-idempotent.  
\end{lemma}
\begin{proof}
    Let $$\kappa := \Big\langle\big\{(x, x^2) \mid x \in R\big\}\Big\rangle.$$ Then $R/\kappa$ is multiplicatively idempotent and $\kappa$ is the least congruence on $R$ with this property. Since $R/\kappa$ is again finitely generated and commutative, Lemma \ref{finitely_generated_commutative_multiplicatively_idempotent_semiring_is_finite} implies that $R/\kappa$ is finite. Since $R$ is finitely generated and $R/\kappa$ is finite, the congruence $\kappa$ is compact in $\mathrm{Con}(R)$ \cite[Theorem 1]{RivalSands}. Now consider the set $$\mathcal{A} := \big\{\theta\in\mathrm{Con}(R) \mid R/\theta \text{ is not multiplicatively idempotent}\big\}$$ ordered by inclusion. This set is non-empty because $\mathrm{id}_{R}\in \mathcal{A}$. Let $(\theta_i)_{i\in I}$ be a chain in $\mathcal{A}$. Put $\theta := \bigcup_{i\in I}\theta_i$. Then $\theta\subseteq R \times R$ is a congruence. We claim that $R/\theta$ is not multiplicatively idempotent. Assume towards a contradiction that $R/\theta$ is multiplicatively idempotent. Then $\kappa \subseteq \theta$ and since $\kappa$ is compact in $\mathrm{Con}(R)$, we have $\kappa \subseteq \theta_i$ for some $i\in I$. Therefore, $R/\theta_i$ is multiplicatively idempotent, contradicting $\theta_i\in \mathcal{A}$. We conclude that $R/\theta$ is not multiplicatively idempotent, so $\theta\in\mathcal{A}$. Thus, every chain in $\mathcal{A}$ has an upper bound in $\mathcal{A}$. By Kuratowski-Zorn Lemma, the set $\mathcal{A}$ has a maximal element $\overline{\theta}$. Put $T := R/\overline{\theta}$. Then $T$ is finitely generated, commutative and not multiplicatively idempotent. Moreover, the maximality of $\overline{\theta}$ implies that every proper quotient of $T$ is multiplicatively idempotent. Hence $T$ is minimal non-idempotent. 
\end{proof}

\subsection{Multiplicatively cancellative and non-cancellative elements}
Let $R$ be a commutative semiring. For $r\in R$, define  $$\Theta_R(r) := \big\{(x, y) \in R\times R\mid rx = ry\big\}.$$ This is the equivalence relation induced by the multiplication map $$m_r : R \to R, \quad x \mapsto rx.$$ Although $m_r$ need not be a semiring homomorphism, the relation $\Theta_R(r)$ is a semiring congruence. We also define $$N_R := \big\{r \in R\mid \Theta_R(r) \neq \mathrm{id}_R\big\}, \quad C_R := R\setminus N_R.$$ Equivalently, $N_R$ consists of those elements $r \in R$ for which the multiplication map $m_r$ is not injective, while $C_R$ consists of those elements $r\in R$ for which the map $m_r$ is injective. Thus, $C_R$ is the set of multiplicatively cancellative elements of $R$. 

\begin{prop}\label{basic_properties_of_N_R_and_C_R}
    Let $(R, +, \cdot)$ be a commutative semiring. 
    \begin{enumerate}[label=(\arabic*)]
        \item If $R$ is minimal non-idempotent, then $N_R$ is either empty or an ideal of $R$. Moreover, $$rs\in N_R \implies r \in N_R\text{ or } s\in N_R$$ for all $r, s \in R$. In particular, if $N_R \neq \varnothing$ and $C_R \neq \varnothing$, then $N_R$ is a prime ideal of $R$. 
        \item If $R$ is multiplicatively divisible and $C_R \neq \varnothing$, then $(C_R, \cdot)$ is a divisible commutative semigroup.
    \end{enumerate}
\end{prop}
\begin{proof}
    $(1)$ Let $N_R \neq \varnothing$. We first prove that $N_R$ is closed under addition. Let $a, b\in N_R$. Then $\Theta_R(a) \neq \mathrm{id}_R$ and $\Theta_R(b) \neq \mathrm{id}_R$. Since $R$ is minimal non-idempotent, there exists $u\in R$ such that $u^2 \neq u$ and the quotients $R/\Theta_R(a)$, $R/\Theta_R(b)$ are multiplicatively idempotent. Thus, $(u, u^2)$ belongs to both $\Theta_R(a)$ and $\Theta_R(b)$, so $au = au^2$ and $bu = bu^2$. Therefore, $$(a + b)u = au + bu = au^2 + bu^2 = (a + b)u^2.$$ Since $u \neq u^2$, the multiplication map $m_{a + b}$ is not injective and in consequence $a + b \in N_R$. 

    Now let $a \in N_R$ and $r \in R$. There exist distinct elements $x, y\in R$ such that $ax = ay$. Multiplying this equality by $r$ and using associativity, we get $(ra)x = (ra)y$ and hence $ra \in N_R$. Therefore, $N_R$ is an ideal of $R$. Now fix $r, s\in R$ and suppose that $rs\in N_R$. If both $r, s$ belonged to $C_R$, then multiplication maps $m_r, m_s$ would be injective and in consequence $m_{rs} = m_r \circ m_s$ would also be injective. This is a contradiction with $rs \in N_R$. Thus, $r\in N_R$ or $s\in N_R$. In particular, if $N_R \neq \varnothing$ and $C_R \neq \varnothing$, then $N_R$ is a non-empty proper subset of $R$ and in consequence it is a prime ideal, as required. \\
    $(2)$ Clearly $C_R$ is closed under multiplication. Let $a\in C_R$ and let $n\geqslant 1$. Since $R$~is multiplicatively divisible, there exists $r\in R$ such that $r^n = a$. If $r \in N_R$, then $r^n \in N_R$, which leads to a contradiction with $a \in C_R$. Hence we conclude that $r \in C_R = R\setminus N_R$. 
\end{proof}

\begin{lemma}\label{every_element_in_N_R_is_multiplicatively_idempotent}
    Let $R$ be a commutative, multiplicatively divisible, minimal non-idempotent semiring. Then for all $a\in N_R$, $b\in C_R$, $$a^2 = a, \quad ba = a.$$ 
\end{lemma}
\begin{proof}
    If $N_R = \varnothing$, then the lemma is vacuously true. Assume now that $N_R\neq \varnothing$ and let $a \in N_R$. Since $R$ is multiplicatively divisible, there exists an element $r\in R$ such that $r^2 = a \in N_R$. By Proposition \ref{basic_properties_of_N_R_and_C_R}(1), we have $r\in N_R$ and in consequence $\Theta_R(r) \neq \mathrm{id}_R$. Since $R$ is minimal non-idempotent, $R/\Theta_R(r)$ is multiplicatively idempotent. We obtain $(r, r^2)\in\Theta_R(r)$ and hence $r^2 = r^3$. Multiplying this equality by $r$, we get $r^3 = r^4$. Therefore, $$a = r^2 = r^3 = r^4 = a^2.$$

    Now let $a\in N_R$ and $b \in C_R$. Then $\Theta_R(a)$ is non-trivial, $R/\Theta_R(a)$ is multiplicatively idempotent and in consequence $(b, b^2)\in\Theta_R(a)$. We get $ab = ab^2$, so $ba = b(ba)$. Since $b\in C_R$, the equality $a = ba$ holds. 
\end{proof}
\begin{lemma}\label{N_R_is_finite}
    Let $R$ be a finitely generated, commutative, multiplicatively divisible, minimal non-idempotent semiring. Then the set $N_R$ is finite. 
\end{lemma}
\begin{proof}
    If $N_R = \varnothing$, then $N_R$ is finite. Assume now that $N_R$ has at least one element. By Proposition \ref{basic_properties_of_N_R_and_C_R}(1), $N_R$ is an ideal of $R$ and hence $m_r(a) = ra \in N_R$ for all $r\in R$ and $a \in N_R$. Thus, the restriction $$\left.m_r\right|_{N_R} : N_R \to N_R, \quad a \mapsto ra$$ is well defined. Let $$\mathrm{Mult}_{N_R}(R) := \big\{\left.m_r\right|_{N_R} \mid r\in R\big\}.$$ We regard $\mathrm{Mult}_{N_R}(R)$ as a semiring with pointwise addition and composition as multiplication. It is easy to see that the map $$\pi : R \to \mathrm{Mult}_{N_R}(R), \quad r \mapsto \left.m_r\right|_{N_R}$$ is a surjective semiring homomorphism. Let $u\in R$ be a non-idempotent element. By Lemma \ref{every_element_in_N_R_is_multiplicatively_idempotent}, we get $u\in C_R$. Since $C_R$ is closed under multiplication, we also have $u^2 \in C_R$. Using again Lemma \ref{every_element_in_N_R_is_multiplicatively_idempotent}, we obtain $$\left.m_u\right|_{N_R} = \left.m_{u^2}\right|_{N_R} = \mathrm{id}_{N_R}.$$ Since $u \neq u^2$, the homomorphism $\pi$ is not injective. Thus, $\mathrm{Mult}_{N_R}(R) \cong R/\ker(\pi)$ is a proper quotient of $R$. Since $R$ is minimal non-idempotent, $\mathrm{Mult}_{N_R}(R)$ is multiplicatively idempotent. Moreover, $\mathrm{Mult}_{N_R}(R)$ is finitely generated. Hence, by Lemma \ref{finitely_generated_commutative_multiplicatively_idempotent_semiring_is_finite}, $\mathrm{Mult}_{N_R}(R)$ is finite. 

    It remains to show that $N_R$ embeds into $\mathrm{Mult}_{N_R}(R)$. Consider the restricted map $\left.\pi\right|_{N_R} : N_R \to \mathrm{Mult}_{N_R}(R)$. Let $a, b\in N_R$. If $\left.m_a\right|_{N_R} = \left.m_b\right|_{N_R}$, then by Lemma \ref{every_element_in_N_R_is_multiplicatively_idempotent}, $$a = a^2 = \left.m_a\right|_{N_R}(a) = \left.m_b\right|_{N_R}(a) = ba = ab = \left.m_a\right|_{N_R}(b) = \left.m_b\right|_{N_R}(b) = b^2 = b.$$ Therefore, the map $\left.\pi\right|_{N_R}$ is injective and in consequence $N_R$ is finite. 
\end{proof}

\begin{prop}\label{every_element_satisfies_one_of_two_alternatives}
Let $R$ be a finitely generated, commutative, multiplicatively divisible, minimal non-idempotent semiring. Then for every element $a\in N_R$ exactly one of the following alternatives holds:   
\begin{enumerate}[label=(\arabic*)]
    \item $C_R + a = \{c\}$, where $c\in N_R$,
    \item $b + a = b$ for every $b\in C_R$.
\end{enumerate}
\end{prop}
\begin{proof}
    Since $R$ is not multiplicatively idempotent, $C_R \neq \varnothing$ by Lemma \ref{every_element_in_N_R_is_multiplicatively_idempotent}. Moreover, $N_R \cap C_R = \varnothing$. Thus, the two alternatives are mutually exclusive. 

    Let $a\in N_R$ and $b\in C_R$. Suppose first that there exists $d\in C_R$ satisfying $d + a =: c \in N_R$. By Lemma \ref{every_element_in_N_R_is_multiplicatively_idempotent}, we have \begin{equation}\label{computations_using_properties_of_N_R_and_C_R}
        c = d + a = b(d + a) = bd + ba = bd + a = db + da = d(b + a).
    \end{equation}
    Since $c\in N_R$, $d\in C_R$ and $C_R$ is closed under multiplication, we get $b + a\in N_R$. Continuing \eqref{computations_using_properties_of_N_R_and_C_R} and using Lemma \ref{every_element_in_N_R_is_multiplicatively_idempotent} again, we obtain $$c = d(b + a) = b + a.$$ Hence $C_R + a = \{c\}$. 

    It remains to handle the case $C_R + a \subseteq C_R$. If $b\in C_R$, then $b + a\in C_R$ and hence $$(b + a)a = a$$ by Lemma \ref{every_element_in_N_R_is_multiplicatively_idempotent}. On the other hand, we have $$(b + a)a = ba + a^2 = a + a,$$ so $a = a + a$. Therefore, $$(b + a)(b + a) = b^2 + 2ab + a^2 = b^2 + 3a = b^2 + a = b^2 + ab = (b + a)b.$$ Since $b + a\in C_R$, the map $m_{b + a}$ is injective and hence $(b + a) = b$. 
\end{proof}

Define $U_R := \{a \in N_R \mid C_R + a \subseteq N_R\}$ and $E_R := N_R\setminus U_R$. 

\begin{corollary}\label{basic_properties_of_U_R_and_E_R}
    Let $R$ be a finitely generated, commutative, multiplicatively divisible, minimal non-idempotent semiring. Then 
    \begin{enumerate}[label=(\arabic*)]
        \item $R + U_R \subseteq U_R$, 
        \item $b + e = b$ for every $b \in C_R$ and $e \in E_R$,
        \item the set $E_R$ is finite, closed under addition and multiplication and every element of $E_R$ is additively idempotent.  
    \end{enumerate}
\end{corollary}
\begin{proof}
    $(1)$ Let $u\in U_R$ and let $r\in R$. We first show that $r + u \in N_R$. If $r\in N_R$, this follows from the fact that $N_R$ is an ideal of $R$, so in particular $N_R$ is closed under addition. If $r\in C_R$, then $r + u \in N_R$ by the definition of $U_R$. 

    It remains to show that $C_R + (r + u) \subseteq N_R$. Let $b \in C_R$. If $b + r \in N_R$, then $$b + (r + u) = (b + r) + u \in N_R,$$ because $N_R$ is closed under addition. If $b + r\in C_R$, then $$b + (r + u) = (b + r) + u \in N_R,$$ because $u\in U_R$. \\
    $(2)$ It follows clearly from definition of $E_R$ and Proposition \ref{every_element_satisfies_one_of_two_alternatives}. \\
    $(3)$ The set $E_R$ is finite, because $E_R \subseteq N_R$ and $N_R$ is finite by Lemma \ref{N_R_is_finite}. Let $e, f \in E_R$. Since $N_R$ is an ideal of $R$, we have $e + f, ef \in N_R$. Let $b \in C_R$. Using $(2)$ twice gives $$b + (e + f) = (b + e) + f = b + f = b$$ and hence $e + f\in E_R$. Now we show that $E_R$ is closed under multiplication. Suppose, for the sake of contradiction, that $ef \in U_R = N_R\setminus E_R$. By $(2)$, we have $b + f = b$ and multiplying this equality by $e$, we get $$be + ef = be.$$ By Lemma \ref{every_element_in_N_R_is_multiplicatively_idempotent}, we have $be = e$. Therefore, $$e + ef = e$$ and hence $$(b + ef) + e = b + (ef + e) = b + e = b \in C_R.$$ Since $ef \in U_R$ and $b\in C_R$, we have $b + ef \in N_R$. Since also $e \in N_R$ and $N_R$ is closed under addition, it follows that $(b + ef) + e \in N_R$, a contradiction. Thus, $ef \in E_R$, so $E_R$ is closed under multiplication. By $(2)$ we have $b + e = b$ and multiplying this equality by $e$, we get $$be + e^2 = be.$$ Then the equality $e + e = e$ follows from Lemma \ref{every_element_in_N_R_is_multiplicatively_idempotent}. Thus, every element of $E_R$ is additively idempotent. This completes the proof. 
\end{proof}

\subsection{Localization at multiplicatively cancellative elements}

\begin{construction}\label{construction_localization_for_commutative_semiring}
    Let $(R, +, \cdot)$ be a commutative semiring and let $\varnothing \neq D \subseteq R$ be a multiplicative subsemigroup. Assume that every element of $D$ is multiplicatively cancellative in $R$. Define a relation $\tau\subseteq (R \times D)\times (R\times D)$ by $$\big((r, d), (s, e)\big)\in\tau \iff re = ds.$$ Let $R[D^{-1}] := (R \times D)/\tau$. Denote by $\frac{r}{d}$ the equivalence class of $(r, d)$ in $R[D^{-1}]$. Now define addition and multiplication on $R[D^{-1}]$ by $$\dfrac{r}{d} + \dfrac{s}{e} := \dfrac{re + sd}{de}, \quad \dfrac{r}{d} \cdot \dfrac{s}{e} := \dfrac{rs}{de}.$$ This makes $(R[D^{-1}], +, \cdot)$ a commutative semiring. Its multiplicative identity is given by $1 := \frac{d}{d}$ for any $d\in D$. 
\end{construction}

From now on, we identify $R$ with its image in $R[D^{-1}]$. 

\begin{lemma}\label{properties_of_sets_after_localization}
    Let $R$ be a commutative, multiplicatively divisible, minimal non-idempotent semiring such that $N_R \neq \varnothing$. Let $\varnothing \neq D\subseteq C_R$ be a multiplicative subsemigroup and let $$C_{R, D} := \bigg\{\dfrac{c}{d} \Bigm\mid c \in C_R, d \in D\bigg\} \subseteq R[D^{-1}].$$ Then 
    \begin{enumerate}[label=(\arabic*)]
        \item $R[D^{-1}] = N_R \sqcup C_{R,D}$,
        \item for all $u\in U_R$, $e \in E_R$, $a \in N_R$ and $b\in C_{R, D}$, $$u + b \in U_R, \quad b + e = b, \quad ba = a.$$
    \end{enumerate} 
\end{lemma}
\begin{proof}
    $(1)$ Let $a \in N_R$ and $d \in D$. Since $d \in C_R$, Lemma \ref{every_element_in_N_R_is_multiplicatively_idempotent} gives $ad = a = ad^2$. Hence we obtain $$\dfrac{a}{d} = \dfrac{ad}{d} = a$$ in $R[D^{-1}]$. Therefore, every element of $R[D^{-1}]$ is either an element of $N_R$ or belongs to $C_{R,D}$. Now suppose that $a = \frac{c}{d}$ for some $c\in C_R$. Then $ad^2 = cd$ and in consequence $a = cd$. But $a\in N_R$, $cd \in C_R$ and $N_R \cap C_R = \varnothing$, a contradiction. Thus, $N_R$ and $C_{R,D}$ are disjoint. \\
    $(2)$ It follows immediately from Lemma \ref{every_element_in_N_R_is_multiplicatively_idempotent} and Corollary \ref{basic_properties_of_U_R_and_E_R}. 
\end{proof}

\begin{lemma}\label{form_of_elements_in_C_T}
    Let $R$ be a finitely generated, commutative, multiplicatively divisible, minimal non-idempotent semiring. Let $X$ be a finite generating set of $R$. Then every element $b \in C_R$ can be written as a finite sum $$b = p_1 + \ldots + p_r,$$ where each $p_i$ is a non-empty finite product of elements of $X \cap C_R$.
\end{lemma}
\begin{proof}
    Let $b\in C_R$. The element $b$ may be written as a finite sum 
    \begin{equation}
        b = p_1 + \ldots + p_r,   
    \end{equation}
    where each $p_i$ is a non-empty finite product of elements of $X$. If all products $p_1, \ldots, p_r$ belong to $E_R$, then $b \in E_R\subseteq N_R$, because $E_R$ is closed under addition. This is a contradiction with $b \in C_R$. Take $k \in \{1, \ldots, r\}$ and consider a product $p_k$. Suppose that at least one of its factors belongs to $N_R$. By Proposition \ref{basic_properties_of_N_R_and_C_R}(1), $N_R$ is an ideal of $R$, so $p_k\in N_R$. If $p_k\in U_R$, then by Corollary \ref{basic_properties_of_U_R_and_E_R}(1), $b \in U_R\subseteq N_R$. This is again a contradiction. Thus, all of the products $p_1, \ldots, p_r$ belong to $C_R \cup E_R$ and at least one of these products belongs to $C_R$. By Corollary \ref{basic_properties_of_U_R_and_E_R}(2), all products lying in $E_R$ can be deleted. The remaining expression is a finite sum of finite products of elements in $X \cap C_R$, as required. 
\end{proof}

\begin{theorem}\cite[Conjecture 2]{KepkaKorbelarLandsmann}\label{solution_to_conjecture_2}
    Let $R$ be a finitely generated commutative semiring. If $R$ is multiplicatively divisible, then $R$ is multiplicatively idempotent. 
\end{theorem}
\begin{proof}
    Suppose, for the sake of contradiction, that the semiring $R$ is multiplicatively divisible and not multiplicatively idempotent. By Lemma \ref{existence_of_minimal_non_idempotent_quotient}, $R$ has a quotient $T$, which is finitely generated, commutative and minimal non-idempotent. Moreover, multiplicative divisibility of $R$ implies multiplicative divisibility of $T$. If $N_T = \varnothing$, then every multiplication map $m_t:T\to T$ is injective and hence $T$ is multiplicatively cancellative. Then by Proposition \ref{finitely_generated_commutative_multiplicatively_divisible_multiplicatively_cancellative_semiring_is_trivial}, $T$ is trivial and in consequence multiplicatively idempotent, a contradiction. Otherwise, $N_T \neq \varnothing$ and hence by Lemma \ref{every_element_in_N_R_is_multiplicatively_idempotent}, $C_T$ contains a non-idempotent element. 

    Fix a finite generating set $X$ of $T$. Assume that $C_T + C_T \subseteq U_T$. Let $b \in C_T$. By Lemma \ref{form_of_elements_in_C_T}, we may write $$b = p_1 + \ldots + p_r,$$ where each $p_i$ is a finite product of elements of $X \cap C_T$. By Proposition \ref{basic_properties_of_N_R_and_C_R}(2), $C_T$ is a divisible commutative semigroup. Hence $p_i\in C_T$ for all $i = 1, \ldots, r$. We claim that $r = 1$. Indeed, if $r \geqslant 2$, then $p_1 + p_2\in U_T$. By Corollary \ref{basic_properties_of_U_R_and_E_R}(1), $T + U_T \subseteq U_T$ and hence $p_1 + \ldots + p_r\in U_T$, contradicting $b \in C_T$. Therefore, every element of $C_T$ is a product of elements of the finite set $X \cap C_T$. Consequently, the divisible commutative semigroup $C_T$ is finitely generated. Thus, by Remark \ref{finite_divisible_semigroup_is_a_finite_band}, the semigroup $C_T$ is a finite band. Hence every element of $C_T$ is multiplicatively idempotent. Since every element of $N_T$ is multiplicatively idempotent by Lemma \ref{every_element_in_N_R_is_multiplicatively_idempotent}, all elements of $T = N_T \cup C_T$ are multiplicatively idempotent, which leads to a contradiction. 

    It remains to handle the case $C_T + C_T \not\subseteq U_T$. First suppose that $$(C_T + C_T)~\cap~E_T~\neq~\varnothing.$$ Then there exist $b, d\in C_T$ such that $b + d = e\in E_T$. Put $D := \{b, b^2, b^3, \ldots\}\subseteq C_T$ and apply Construction \ref{construction_localization_for_commutative_semiring} to obtain a commutative semiring $T[D^{-1}]$. Consider a set $$C_{T, D} := \bigg\{\dfrac{c}{b^n} \Bigm\mid c \in C_T, n \geqslant 1\bigg\} \subseteq T[D^{-1}].$$ Then we have $$1 + \dfrac{d}{b} = \dfrac{b}{b} + \dfrac{d}{b} = \dfrac{b^2 + db}{b^2} = \dfrac{b(b + d)}{b^2} = \dfrac{be}{b^2} = e$$ and in consequence \begin{equation}\label{some_helpful_equality_after_localization}
        c + c\cdot \dfrac{d}{b} = c\bigg(1 + \dfrac{d}{b}\bigg) = ce = e
    \end{equation}
    for every $c \in C_{T, D}$, where the last equality follows from Lemma \ref{properties_of_sets_after_localization}(2). We claim that $P := C_{T, D}\cup\{e\}$ is a commutative ring. Let $c, g \in C_{T, D}$. By Lemma \ref{properties_of_sets_after_localization}(1), we have $c + g \in C_{T, D} \sqcup U_T \sqcup E_T$. If $c + g\in U_T$, then $(c + g) + c\cdot \frac{d}{b}\in U_T$. On the other hand, using \eqref{some_helpful_equality_after_localization}, $$(c + g) + c\cdot\dfrac{d}{b} = g + \bigg(c + c\cdot \dfrac{d}{b}\bigg) = g + e = g \in C_{T, D},$$ which leads to a contradiction. If $c + g \in E_T$, then the same computation gives $(c + g) + c\cdot\frac{d}{b} = g \in C_{T, D}$. But by Lemma \ref{properties_of_sets_after_localization}(2), we have $(c + g) + c\cdot\frac{d}{b} = c\cdot\frac{d}{b}$. Hence $g = c\cdot\frac{d}{b}$ and then using \eqref{some_helpful_equality_after_localization}, $$c + g = c + c\cdot\frac{d}{b} = e.$$ We conclude that $C_{T, D} + C_{T, D} \subseteq C_{T, D} \cup \{e\} = P$. The element $e$ is an additive identity of $P$ and by Corollary \ref{basic_properties_of_U_R_and_E_R}(3), it is additively idempotent. Moreover, by \eqref{some_helpful_equality_after_localization} every element of $P$ has an additive inverse. It can be easily verified that $P \cdot P \subseteq P$, so $P$ is a commutative ring. 

    By Lemma \ref{form_of_elements_in_C_T}, every element of $C_T$ is a finite sum of products of elements of $X \cap C_T$. Hence $P$ is finitely generated as a semiring by $(X \cap C_T) \cup\{\frac{1}{b}, e\}.$ It is also multiplicatively divisible. Indeed, let $\frac{c}{b^m} \in C_{T, D}$ and let $n \geqslant 1$. Since $T$ is multiplicatively divisible, there exists an element $t \in T$ such that $$t^n = cb^{m(n - 1)}.$$ The right-hand side belongs to $C_T$, so also $t^n\in C_T$. If $t\in N_T$, then $t^n\in N_T$, a contradiction. Hence $t \in C_T$, and therefore $\frac{t}{b^m}\in C_{T, D}$. Moreover, $$\bigg(\dfrac{t}{b^m}\bigg)^n = \dfrac{c}{b^m}.$$ Since $e^2 = e$, we have $e^n = e$. Thus, $P$ is a finitely generated, commutative, additively cancellative semiring which is multiplicatively divisible. By Proposition \ref{finitely_generated_commutative_multiplicatively_divisible_additively_cancellative_semiring_is_multiplicatively_idempotent}, $P$ is multiplicatively idempotent. In particular, every element of $C_T \subseteq P$ is multiplicatively idempotent. By Lemma \ref{every_element_in_N_R_is_multiplicatively_idempotent}, every element of $N_T$ is also multiplicatively idempotent. Hence $T = N_T \cup C_T$ is multiplicatively idempotent, which leads to a contradiction. 

    Now suppose that $$(C_T + C_T) \cap E_T = \varnothing.$$ Consequently, we have $C_T + C_T \subseteq C_T \cup U_T$. By assumption $C_T + C_T \not\subseteq U_T$, so there exist $x, y \in C_T$ such that $x + y \in C_T$. If $C_T + C_T \subseteq C_T$, then $C_T$ is a subsemiring of $T$. By Lemma \ref{form_of_elements_in_C_T}, it is finitely generated as a semiring. It is multiplicatively divisible by Proposition \ref{basic_properties_of_N_R_and_C_R}(2), and multiplicatively cancellative by definition of $C_T$. Hence Proposition \ref{finitely_generated_commutative_multiplicatively_divisible_multiplicatively_cancellative_semiring_is_trivial} implies that $C_T$ is trivial. This contradicts the fact established above that $C_T$ contains a non-idempotent element. Thus, it remains to handle the case in which both intersections $(C_T + C_T)\cap C_T$ and $(C_T + C_T)\cap U_T$ are non-empty. Since $(C_T, \cdot)$ is a cancellative commutative semigroup, we apply Construction \ref{construction_of_completion_group} to obtain the abelian group $(G(C_T), \cdot)$. Define $$P := \bigg\{\dfrac{p}{q}\in G(C_T) \Bigm\mid p, q \in C_T, \ p + q \in U_T\bigg\}.$$ Then $P \neq \varnothing$. We first observe that no element of $G(C_T)$ can be represented both in the form $$\dfrac{p}{q},\qquad p,q\in C_T, \quad p + q \in U_T,$$ and in the form $$\dfrac{b}{d}, \qquad b, d \in C_T, \quad b + d \in C_T.$$ Indeed, suppose that $\frac{p}{q} = \frac{b}{d}$. Then $bq = pd$ and in consequence \begin{equation}\label{equality_leading_to_contradiction}
        b(p + q) = bp + bq = bp + pd = p(b + d).
    \end{equation}
    Since $p + q \in U_T$ and $b\in C_T$, we have $b(p + q) = p + q\in U_T$ by Lemma \ref{every_element_in_N_R_is_multiplicatively_idempotent}. Therefore, the left-hand side of \eqref{equality_leading_to_contradiction} belongs to $U_T$, whereas the right-hand side belongs to $C_T$, a contradiction. Now, since $(C_T + C_T)\cap C_T \neq \varnothing$, there exist $b, d \in C_T$ such that $b + d \in C_T$. By the observation above, we get $\frac{b}{d}\not\in P$ and hence $P$ is a proper subset of $G(C_T)$. 

    Put $K := \{x \in G(C_T) \mid Px = P\}$. Then $K$ is a subgroup of $G(C_T)$. We claim that if $x, y, w \in C_T$ and $x + y = w$, then $\frac{w}{x}, \frac{w}{y}\in K$. Take elements $p, q\in C_T$ such that $p + q \in U_T$. Then we have $$(p + q)x + qy = px + qx + qy = px + q(x + y) = px + qw.$$ By Lemma \ref{every_element_in_N_R_is_multiplicatively_idempotent}, $(p + q)x \in U_T$ and hence by Corollary \ref{basic_properties_of_U_R_and_E_R}(1), the left-hand side belongs to $U_T$. Then $px + qw$ also belongs to $U_T$ and in consequence $\frac{p}{q}\cdot \frac{x}{w}=~\frac{px}{qw}~\in~P$. Hence $P \cdot \frac{x}{w} \subseteq P$. Similarly, we show that $P\cdot \frac{w}{y} \subseteq P$. Interchanging $x$ and $y$ gives the reverse inclusions, so $\frac{w}{x}, \frac{w}{y} \in K$, as required. 

    We now prove that $K = G(C_T)$. Let $a \in C_T$. By Lemma \ref{form_of_elements_in_C_T}, we have $a = p_1 + \ldots + p_r$, where each $p_i$ is a product of elements of $X \cap C_T$. Consider the partial sums $s_1 = p_1$ and $s_{i+1} = s_i + p_{i + 1}$ for all $i = 1, \ldots, r - 1$. Since $a \in C_T$ and $T + U_T\subseteq U_T$, no partial sum can belong to $U_T$. Since $C_T + C_T \subseteq C_T \cup U_T$, all partial sums belong to $C_T$. Then $\frac{s_{i+1}}{s_i}\in K$ for all $i = 1, \ldots, r - 1$ and in consequence $s_{i+1}K = s_iK$. We obtain $$aK = s_rK = s_{r - 1}K = \ldots = s_1K = p_1K.$$ Hence the image of $C_T$ in $G(C_T)/K$ is generated, as a multiplicative semigroup, by the images of the elements of $X \cap C_T$. Therefore, it is a finitely generated divisible commutative semigroup. By Remark \ref{finite_divisible_semigroup_is_a_finite_band}(1), the image of $C_T$ in $G(C_T)/K$ is a finite band. But $G(C_T)/K$ is a group and hence the image of $C_T$ is trivial in $G(C_T)/K$. Thus, $C_T\subseteq K$ and in consequence $K = G(C_T)$, as required. But then $PG(C_T) = P$. Since $P$ is non-empty, $PG(C_T) = G(C_T)$ and hence $P = G(C_T)$, contradicting the fact that $P$ is a proper subset of $G(C_T)$. This contradiction eliminates the final remaining case. Hence no finitely generated, commutative, multiplicatively divisible, minimal non-idempotent semiring $T$ exists. This contradicts the choice of $T$ and proves that $R$ is multiplicatively idempotent. 
\end{proof}
\bibliographystyle{plain}
\bibliography{references}
\end{document}